\documentclass[12pt]{amsart}
\usepackage{amsmath,amssymb,amsbsy,amsfonts,amsthm,latexsym,
                        amsopn,amstext,amsxtra,euscript,amscd,mathrsfs,color}

\usepackage{float}
\usepackage[english]{babel}
\usepackage{url}
\usepackage[colorlinks,linkcolor=blue,anchorcolor=blue,citecolor=blue]{hyperref}
\usepackage{graphicx}

\begin{document}

\newtheorem{theorem}{Theorem}
\newtheorem{lemma}[theorem]{Lemma}
\newtheorem{claim}[theorem]{Claim}
\newtheorem{corollary}[theorem]{Corollary}
\newtheorem{proposition}[theorem]{Proposition}
\newtheorem{definition}[theorem]{Definition}
\newtheorem{question}[theorem]{Question}
\newtheorem{remark}[theorem]{Remark}
\newtheorem{example}[theorem]{Example}

\newcommand{\F}{\mathbb{F}}
\newcommand{\K}{\mathbb{K}}
\newcommand{\D}[1]{D\(#1\)}
\newcommand{\scr}{\scriptstyle}
\renewcommand{\\}{\cr}
\renewcommand{\(}{\left(}
\renewcommand{\)}{\right)}
\renewcommand{\[}{\left[}
\renewcommand{\]}{\right]}
\newcommand{\<}{\langle}
\renewcommand{\>}{\rangle}
\newcommand{\fl}[1]{\left\lfloor #1 \right\rfloor}
\newcommand{\rf}[1]{\left\lceil#1\right\rceil}
\renewcommand{\le}{\leqslant}
\renewcommand{\ge}{\geqslant}
\newcommand{\eps}{\varepsilon}
\newcommand{\mand}{\qquad\mbox{and}\qquad}
\newcommand{\cC}{\mathcal{C}}

\newcommand{\comm}[1]{\marginpar{%
\vskip-\baselineskip %raise the marginpar a bit
\raggedright\footnotesize
\itshape\hrule\smallskip#1\par\smallskip\hrule}}

\newcommand{\commI}[1]{\marginpar{%
\vskip-\baselineskip %raise the marginpar a bit
\raggedright\footnotesize
\itshape\hrule\smallskip\color{blue}#1\par\smallskip\hrule}}

\newcommand{\Fq}{\mathbb{F}_q}
\newcommand{\Fp}{\mathbb{F}_p}
\newcommand{\Z}{\mathbb{Z}}
\newcommand{\cS}{\mathcal{S}}
\newcommand{\Disc}[1]{\mathrm{Disc}\(#1\)}
\newcommand{\Res}[1]{\mathrm{Res}\(#1\)}

\newcommand{\Nm}[1]{\mathrm{Norm}_{\,\F_{q^k}/\Fq}(#1)}

\def\Tr{\mbox{Tr}}
\newcommand{\rad}[1]{\mathrm{rad}\(#1\)}
\newcommand\al{\alpha}
\newcommand\be{\beta}
\newcommand\ga{\gamma}

\numberwithin{equation}{section}
\numberwithin{theorem}{section}

\title[Consecutive Polynomial Sequences]{The Arithmetic of Consecutive Polynomial Sequences over Finite Fields}

 \author{Domingo G\'omez-P\'erez}
 \address{Department of Mathematics, Statistics and Computer Science, University of Cantabria, Santander 39005, Spain}
 \email{domingo.gomez@unican.es}
\author{Alina Ostafe}
\address{School of Mathematics and Statistics, University of New South Wales, Sydney, NSW 2052, Australia}
\email{alina.ostafe@unsw.edu.au}

\author{Min Sha}
\address{School of Mathematics and Statistics, University of New South Wales, Sydney, NSW 2052, Australia}
\email{shamin2010@gmail.com}


%\date{Januar1 1, 2001 and, in revised form, June 22, 2001.}



\begin{abstract}
Motivated by a question of van der Poorten about the existence of an infinite chain of prime numbers
(with respect to some base), in this paper we advance the study of  sequences of consecutive polynomials whose coefficients are chosen consecutively from a sequence in a finite field of odd prime characteristic. We study the arithmetic of such sequences, including  bounds for the largest degree  of irreducible factors, the number of irreducible factors, as well as for the number of such sequences of fixed length in which all the polynomials are irreducible.
\end{abstract}

\maketitle

\section{Introduction}

\subsection{Motivation}
\label{sec:motivation}
In \cite{Poorten}, van der Poorten observed that the numbers $$19,197,1979,19793, 197933,1979339,19793393,197933933,1979339339$$
are all prime numbers and raised
a question that whether there is such an infinite chain of prime numbers
(with respect to some base $b$).  One related question is whether there exists the largest truncatable prime in a given base $b$ (such a  prime can yield a sequence of primes when digits are removed away from the right). Note that the above integer $1979339339$ is not a truncatable prime. The authors in~\cite{Angell} have given heuristic arguments for
the length of the largest truncatable prime in base $b$ (roughly, the length is $be/\log b$, where $e$ is the base of the natural logarithm) and computed the largest truncatable primes in base $b$ for $3\le b \le 15$.
Both questions might be very hard.

Mullen and Shparlinski~\cite[Problem 31]{MuShp} asked
an analogous question about polynomials over finite fields.
More precisely, let $p$ be an odd prime number and $q=p^s$ for some positive integer $s$.
We denote by $\Fq$ the finite field of $q$ elements,
and use $\Fq[X]$ to denote the ring of polynomials with
coefficients in $\Fq$.

For a (finite or infinite) sequence $\{ u_n \}_{n\ge 0}$, of non-zero elements in $\Fq$, we
define a {\it consecutive polynomial sequence} $\{ f_n \}_{n\ge 1}$, associated to the sequence $\{ u_n \}$,
 in $\Fq[X]$ as follows:
\begin{equation}
\label{eq:conspolyseq}
f_n=u_nX^n+\ldots+u_1X+u_0,\ n\ge 1.
\end{equation}
If all the polynomials $f_n$, $n\ge 1$, are irreducible, then the
sequence $\{ f_n \}$ is called a {\it consecutive irreducible polynomial sequence},
and $\{ u_n \}$ is called a {\it consecutive irreducible sequence}.

Given a sequence $\{ u_n \}$, let $L(\{u_n\})$ be either $\infty$ if $\{u_n\}$ is infinite,
or a non-negative integer such that $L(\{u_n\})+1$ is the length of $\{u_n\}$.
That is, $L(\{u_n\})$ is the length of the associated polynomial sequence $\{f_n\}$.

Mullen and Shparlinski~\cite[Problem 31]{MuShp} asked for lower and upper bounds
for the maximum length $L(q)= \max \{ L(\{u_n\})\}$ (possibly infinite), where $\{u_n\}$ runs through all consecutive irreducible sequences
over $\Fq$. The only known result  is a lower bound due to
Chow and Cohen~\cite[Theorem 1.2]{CC},
\begin{equation}
\label{eq:Lq lower}
L(q) > \frac{\log q}{2\log\log q},
\end{equation}
whenever $q \ne 3$; they also observed that for $q=3$, $L(3)=3$.

The work on irreducible polynomials with prescribed coefficients might
reflect that such an upper bound of $L(q)$ indeed exists. Twenty years ago,
Hansen and Mullen \cite[Conjecture B]{HM} conjectured that for any $n\ge 3$, there exists a monic irreducible polynomial of degree $n$
over $\Fq$ with a prescribed coefficient. This conjecture has been proved by Wan \cite{Wan} and Ham and Mullen \cite{HaM}.
Recently, Ha \cite{Ha} has showed that for any $n \ge 8$ and $0< \epsilon < 1/4$,
there exists an irreducible polynomial of degree $n$ over $\Fq$
with any $\lfloor (1/4 - \epsilon)n \rfloor$ coefficients preassigned (the constant term is non-zero)
  when $q$ is sufficiently large depending on $\epsilon$; see \cite{Pollack} for a previous result.
However, to search for consecutive irreducible sequences, we need to fix $n$ values $u_0,u_1,\ldots,u_{n-1}\in \F_q^*$ and find $u_n \in \F_q^*$
such that the polynomial $u_nX^n+\cdots +u_1 X +u_0$ is irreducible.
Thus, the difficulty of the above work suggests that searching for consecutive irreducible sequences of infinite
length might be hard.
Moreover, in Section \ref{sec:heuristic} we give a heuristic argument to predict that $q \le L(q) < 3q$, which is consistent with the numerical data.

We also want to remark that it is easy to construct an infinite chain of consecutive irreducible polynomials
over the rational integers $\Z$. For example, given a prime number $\ell$, all the polynomials $1+\ell X,1+\ell X+\ell X^2,1+\ell X+\ell X^2+\ell X^3,\ldots$
are irreducible, which can be obtained by using Eisenstein's criterion to their reciprocal polynomials and with respect to the prime number $\ell$.

Throughout the paper, we use the Landau symbols $O$ and $o$ and the Vinogradov symbol $\ll$. We recall that the assertions $U=O(V)$ and $U\ll V$ (sometimes we write this also as $V \gg U$) are both equivalent to the inequality $|U|\le cV$ with some constant $c>0$, while $U=o(V)$ means that $U/V\to 0$. In this paper, the constants implied in the symbols $O, \ll$ are absolute and independent of any parameters. If the implied constant is not absolute and depends on some parameter $\rho$, then we write $O_\rho$ and $\ll_\rho$.

\subsection{Our results and methods}
In this paper, we study the arithmetic of consecutive polynomial sequences, such as, the growth of the largest degree of irreducible factors, the number of irreducible factors, as well as giving upper and lower bounds for the number of consecutive irreducible sequences of fixed length. We describe below our results and the techniques we use in more details.

Let $\{f_n\}$ be a consecutive polynomial sequence.
We first introduce some notation:
\begin{itemize}
\item $D(f_n)$: the largest degree of the irreducible factors of $f_n$;

\item $\omega(f)$: the number of distinct monic irreducible factors of a polynomial $f\in \Fq[X]$;

\item $I_{N}$: the number  of consecutive irreducible polynomial sequences of length $N$.
\end{itemize}

In Section~\ref{sec:prel}, we introduce the main tools that we use to prove our results. In Section~\ref{sec:degFactors}, we
use a method introduced in~\cite{GOS}, relying on the
polynomial $ABC$ theorem (proved first by Stothers~\cite{Stothers}, and then independently by Mason~\cite{Mason1,Mason2} and Silverman~\cite{Si84}, see also~\cite{Sny}),
to give a lower bound for $D(f_n)$, $n\ge 1$.
In particular, we prove that if $\{f_n\}$ is of infinite length, for almost all integers $n\ge 1$ we have
$$
D(f_n)\gg \frac{\log n}{\log q}.
$$

In Section~\ref{sec:numberFactors}, using similar ideas as in~\cite{PShp1}, we prove that
for any integers $m\ge 0$ and $H\ge 2$ we have
$$
\omega(f_{m+1}f_{m+2}\cdots f_{m+H})\gg \frac{(m+H)H}{m+H\log (m+H)},
$$
whenever the polynomials $f_{m+1},f_{m+2},\ldots, f_{m+H}$ are well-defined
(note that the sequence $\{f_n\}$ can be of finite length).

Given a finite set $\cS$ of irreducible polynomials in $\F_q[X]$, we also give an upper bound for the number of $\cS$-polynomials among $H$ consecutive polynomials $f_{m+1},f_{m+2},\ldots, f_{m+H}$.

We conclude this section by showing that there exists a consecutive polynomial sequence $\{ f_n \}$ of length at least $\lfloor \sqrt{2(q-1)} + 3/2\rfloor$
 such that all the polynomials are pairwise coprime. In this setting, the bound is much better than that in \eqref{eq:Lq lower}.

In Section~\ref{sec:numberconsecirreduc}, we give upper and lower bounds for $I_{N}$. This is also the most technical part of the paper. To give such bounds, we use a sieve for large values of $N$ and also the Weil bound for multiplicative character sums, together with Stickelberger's Theorem~\cite{St,Swan} (which gives the parity of irreducible factors of a polynomial) for $N$ that is not too large
compared to $q$. We prove that for any integer $N\ge 2$, we have
\begin{equation}
\label{eq:I_N1}
I_N < 3^{-N/7+1} q^{N+1}
\end{equation}
and
\begin{equation}
\label{eq:I_N2}
I_N < 2^{-N+1}q^{N+1}+N^2 q^{N+1/2}+ N^4 q^{N}.
\end{equation}
Note that \eqref{eq:I_N2} is better than \eqref{eq:I_N1} when $q$ is much larger than $N$.
The rest of the section is dedicated to obtaining a formula for $I_2$ and  explicit lower bounds for $I_3$ and $I_4$, which are better than those implied in \cite{CC}.

Finally, we want to remark that analogues of our results can be considered for sequences of $g$-ary digits,
$g\ge 2$; see \cite{GS}. More precisely, given a sequence of $g$-ary digits $\{d_n\}_{n\ge 0}, d_n \in \{0,1,\ldots, g-1 \}$, we can define a sequence of integers $\{ a_n \}_{n\ge 0}$ by
$$
a_n = \sum_{i=0}^{n} d_i g^i.
$$
Then, one can study arithmetic properties of the integer sequence $\{ a_n \}$.

\section{Preliminaries}
\label{sec:prel}

In this section we gather some tools which are used in the proofs  for the convenience of the reader.

We start by recalling a few properties of discriminants and resultants of polynomials.
A  detailed exposition on this subject can be found in~\cite[Part III, Chapter 15]{Childs}.
For two polynomials $f,g\in\F_q[X]$, we denote by
\begin{itemize}
\item  $\Disc{f}$: the discriminant of  $f$;

\item $\Res{f,g}$: the resultant of $f$ and $g$.
\end{itemize}
The following well-known formula for the discriminant of the product $fg$ can be found in~\cite[Part III, Chapter 15, Proposition 2]{Childs}
(see also \cite[Theorem 3.10]{Janson1}),
\begin{equation}
\label{eq:discprod}
\Disc{fg}=\Disc{f}\Disc{g}\Res{f,g}^2.
\end{equation}

The discriminant of a polynomial $f$ can be viewed as a polynomial function in the coefficients of $f$.
This point of view gives the following simple formula, which can be regarded as a
relation between discriminants of polynomials of consecutive degrees.
Let $f\in\F_q[X]$ of degree at most $d$ be written as
$$
f=a_dX^d+g, \quad g=a_{d-1}X^{d-1}+\cdots+a_1X+a_0.
$$
If we first compute $\Disc{f}$ as a function in $a_0,a_1,\ldots,a_d$ and then set $a_d=0$, we can get the following relation
\begin{equation}
\label{eq:disc n^n-1}
\Disc{f}\Big\rvert_{a_d=0}=a_{d-1}^2\Disc{g};
\end{equation}
see \cite[Theorem 3.11]{Janson1}.

One of the main tools used in our proofs
is Stickelberger's Theorem (see~\cite{St} or~\cite[Corollary 1]{Swan}),
which gives the parity of the number of distinct irreducible factors of a square-free polynomial over a finite field of odd characteristic.
This provides a powerful tool to study the number of irreducible factors of polynomials.

\begin{lemma}
\label{lem:Stick}
  Suppose that $f\in\F_q[X]$, where $q$ is odd, is a polynomial of degree $d\ge 2$ and is
  the product of $r$ pairwise
  distinct irreducible polynomials over $\F_q$.
 Then $r\equiv d\pmod 2$ if and only
  if $\Disc{f}$ is a square element in $\F_q$.
\end{lemma}

For proving our results, we treat the discriminant of a general polynomial $f$ as a multivariate polynomial in
the coefficients of $f$ and study for which substitutions of the variables the discriminant is a square.
This technical result has been given in~\cite[Lemma 5.2]{GNOS}, which in fact implies an explicit result.
Here, we reproduce the proof briefly.

\begin{lemma}
\label{lem:square}
Let $G\in\F_q{[Y_0,Y_1,\ldots,Y_d]}$ be a polynomial of degree $m$, which is not a square polynomial in the algebraic
closure of $\F_q$. Then there exists $i \in \{0,1, \ldots, d\}$ such that
$G(a_0,\ldots,a_{i-1}, Y_{i},a_{i+1},\ldots,a_d)$
is not a square polynomial in $Y_i$ up to a multiplicative constant for all but at most
$m^2q^{d-1}$  values  of $a_0,\ldots,a_{i-1}, a_{i+1},\ldots,a_d \in \F_q$.
\end{lemma}

\begin{proof}
As in the proof of \cite[Lemma 5.2]{GNOS},
let
$$
G(Y_0,\ldots, Y_d) = aG_1(Y_0,\ldots,Y_d)^{d_1} \cdots G_h(Y_0,\ldots, Y_d)^{d_h}
$$
be the decomposition of the polynomial in a product of a non-zero constant and monic irreducible polynomials,
and assume that $d_1$ is an odd integer and $G_1(Y_0,\ldots,Y_d)$ depends on some variable $Y_i$.
The result in \cite[Lemma 5.2]{GNOS} comes from the sum of three upper bounds $mq^{d-1}, \deg G_1 (\deg G_1-1)q^{d-1}$ and $\deg G_1 \deg G_jq^{d-1}$, where $j$ is some integer between 2 and $h$ (it may not exist).

In fact, if the polynomial $G(a_0,\ldots,a_{i-1}, Y_{i},a_{i+1},\ldots,a_d)$ is a constant polynomial under the specialisation $a_0,\ldots,a_{i-1}, a_{i+1},\ldots,a_d \in \F_q$, then for some $k$,
$G_k(a_0,\ldots,a_{i-1}, Y_{i},a_{i+1},\ldots,a_d)$ is also a constant.
So, the bound $mq^{d-1}$ can be replaced by $q^{d-1} \max_{1\le k \le h}\deg G_k $. Noticing that
$$
m^2 \ge \max_{1\le k \le h}\deg G_k + \deg G_1(\deg G_1-1) + \deg G_1 \deg G_j,
$$
we get the desired result.
\end{proof}

To estimate the number of consecutive irreducible sequences,  we need the Weil bound for character sums with polynomial arguments (see~\cite[Theorem 5.41]{LN}).

\begin{lemma}
\label{lem:Weil}
Let $\chi$ be a multiplicative character of $\F_q$ of order $m>1$, and let $f\in\F_q[X]$ be a polynomial of
positive degree that is not, up to a multiplicative constant, an $m$-th power of a polynomial. Let $d$ be the number of distinct
roots of $f$ in its splitting field over $\F_q$. Under these conditions, the following inequality holds:
$$
\left| \sum_{x\in\F_q}\chi(f(x))\right|\le (d-1)q^{1/2}.
$$
\end{lemma}

Some of our results are also based on the polynomial $ABC$ theorem~\cite{Mason1, Mason2,Si84,Sny,Stothers}.

For a non-zero polynomial $f \in \Fq[X]$, we denote by   $\rad f$  the product of all distinct monic irreducible factors of $f$.

\begin{lemma}
\label{lem:ABC}
Let $A$, $B$, $C$ be non-zero polynomials in $\Fq[X]$
with $A+B+C = 0$ and $\gcd\(A, B, C\)= 1$. If
$\deg A \ge \deg \rad {ABC}$, then for their derivatives, we have $A'=B'=C'= 0$.
\end{lemma}

To obtain an upper bound for the number of consecutive irreducible sequences of fixed length,
we need the following result due to Johnsen \cite[Corollary 2]{Johnsen} on the number of irreducible polynomials over $\Fq$ in an arithmetic progression.

\begin{lemma}
\label{lem:Johnsen}
Let $n$ and $r$ be positive integers such that $1\le r <n$, and let $f\in \Fq[X]$.
Then, the number of irreducible polynomials of degree $n$ which are congruent to $f$ modulo $X^r$ is less than
$2q^{n-r+1} / (n-r)$.
\end{lemma}

Finally, we recall a classical result on using the cubic resolvent to solve quartic equations, which is due to Euler \cite[\S 5]{Euler}.
Here, we reproduce a form from \cite[Theorem 3.2]{Janson2}.

\begin{lemma}
\label{lem:resolvent}
Let $K$ be an arbitrary field of characteristic not equal to $2$ or $3$.
Given a quartic polynomial $f(X)=X^4+aX^2+bX+c \in K[X]$, define its cubic resolvent by $R(X)=  X^3 + 2aX^2 + (a^2-4c)X - b^2 $. Let $u,v,w$ be the roots of $R(X)$, and put $\ga_1= \sqrt{u},\ga_2= \sqrt{v},\ga_3= \sqrt{w}$, where we choose the signs so that $\ga_1\ga_2\ga_3= - b$. Then, the roots of $f$ are given by
\begin{equation*}
  \label{eq:h4}
\left\{\begin{array}{ll}
  \be_1 = \frac{1}{2}(\ga_1+\ga_2+\ga_3),\\
  \be_2 = \frac{1}{2}(\ga_1-\ga_2-\ga_3),\\
  \be_3 = \frac{1}{2}(-\ga_1+\ga_2-\ga_3),\\
  \be_4 = \frac{1}{2}(-\ga_1-\ga_2+\ga_3).
\end{array}\right.
\end{equation*}
\end{lemma}

\section{The Largest Degree of Irreducible Factors}
\label{sec:degFactors}

We recall that for a consecutive polynomial sequence $\{ f_n \}$,
we use $D(f_n)$ to denote the largest degree of  irreducible factors of $f_n$ for each $n\ge 1$.

The following is our main result of this section. We use the
same technique as in the proof of~\cite[Theorem 10]{GOS}.
Recall that $p$ is an odd prime and the characteristic of $\Fq$.

\begin{theorem}
\label{thm:Degree}
Let $\{ f_n \}$ be a consecutive polynomial sequence  of infinite length.
For any integers $n\ge 2q-1$ and $d$ satisfying $0<d\le \frac{\log((n+1)/2)}{\log q}$, we have
\begin{equation}
\label{eq:Degree1}
\max\{D(f_n),D(f_{n+d})\} > \frac{\log((n+1)/2)+ \log\log q - \log\log ((n+1)/2)}{\log q}.
\end{equation}
Moreover, if $p \nmid n+1$ or $p \nmid d$, then
\begin{equation}
\label{eq:Degree2}
\max\{D(f_n),D(f_{n+d})\} > \frac{\log((n+1)/2)}{\log q}.
\end{equation}
\end{theorem}

\begin{proof}

Fix an integer $n\ge 2q-1$ and fix an integer $d$ such that
$$
0<d\le \frac{\log((n+1)/2)}{\log q}.
$$
By construction in \eqref{eq:conspolyseq} we have
\begin{equation*}
  f_{n+d} - f_{n} = X^{n+1}\(\sum_{i=1}^{d}u_{i+n}X^{i-1}\).
\end{equation*}
Let $g=\gcd(f_n,f_{n+d})$.
Then we must have that $g$ divides $\sum_{i=1}^{d}u_{i+n}X^{i-1}$,
and so $\deg g \le d-1$.
Put $A = f_{n+d}/g$, $B = -f_{n}/g$
and
$$
C = -X^{n+1}\(\sum_{i=1}^{d}u_{i+n}X^{i-1}\)/g.
$$
Then,
$$
A+B+C=0 \quad \textrm{and \quad $\gcd(A,B,C)=1$}.
$$
Let $m$ be the largest non-negative integer such that $A=A_1^{p^m},B=B_1^{p^m},C=C_1^{p^m}$ for some polynomials $A_1,B_1,C_1$ such that the identity about derivatives $A_1^{\prime}=B_1^{\prime}=C_1^{\prime}=0$ does not hold. Note that $m=0$ if and only if the identity $A^{\prime}=B^{\prime}=C^{\prime}=0$ does not hold. Then, we have
$$
A_1+B_1+C_1=0 \quad \textrm{and \quad $\gcd(A_1,B_1,C_1)=1$}.
$$
By the form of $C$, we can write $C_1$ as
$$
C_1 = X^{(n+1)/p^m} h(X) \quad \textrm{with \quad $\deg h \le (d-1)/p^m$}
$$
for some polynomial $h(X)$ (note that we indeed have $p^m \mid n+1$).

Since both $\deg A$ and $\deg B$ are divisible by $p^m$, we get $p^m \mid d$. So the choice of $d$ implies that
\begin{equation}
\label{eq:p^m}
p^m \le \frac{\log((n+1)/2)}{\log q}.
\end{equation}

We define $N$ as the largest integer satisfying
\begin{equation}
\label{eq:DnB}
2q^N \le (n+1)/p^m.
\end{equation}
So, we have
\begin{equation}
\label{eq:N+1}
N+1 > \frac{\log((n+1)/2)-m\log p}{\log q}.
\end{equation}

If $N=0$, then we obtain
$$
n+1 < 2qp^m \le \frac{2q\log((n+1)/2)}{\log q},
$$
which implies that the right-hand side of \eqref{eq:Degree1} is less than 1, and thus \eqref{eq:Degree1} is true automatically.

In the following we assume that $N\ge 1$.
Now, we prove the desired result by contradiction. Suppose that
$$
\max\{D(f_n),D(f_{n+d})\} \le N.
$$
This means that the polynomial $f_nf_{n+d}$ can be
factorized by irreducible polynomials of degree at most $N$.
So, any root of $f_n$ or $f_{n+d}$ belongs to $\F_{q^{j}}$ with
$j\le N$.
Then, the product $f_{n}f_{n+d}$ has at most
\begin{equation}
\label{eq:fnd}
\sum_{j=1}^N q^j < 2q^N
\end{equation}
distinct roots.

 Then, applying Lemma \ref{lem:ABC} to $A_1,B_1$ and $C_1$, we obtain
 \begin{align*}
   \frac{n+1}{p^m} \le \deg A_1 & < \deg \rad{A_1B_1C_1} \\
   & \le \deg \rad{ f_{n+d}f_{n}Xh(X)}\le 2q^N,
 \end{align*}
 where the last inequality comes from \eqref{eq:fnd} and the fact $\deg h < N$
 (which can be straightforward proved by contradiction and by collecting \eqref{eq:p^m} and \eqref{eq:N+1}
 and noticing the choices of $h(X)$ and $d$, where we can assume that $d \ge 2$).
Hence, we get  $(n+1)/p^m< 2q^N$, which contradicts  \eqref{eq:DnB}.
So, we must have
\begin{equation}
\label{eq:general deg}
\max\{D(f_n),D(f_{n+d})\} \ge N+1 > \frac{\log((n+1)/2)-m\log p}{\log q},
\end{equation}
which, together with \eqref{eq:p^m}, concludes the proof of \eqref{eq:Degree1}.

Now, it remains to prove \eqref{eq:Degree2}.
If the derivatives $A^{\prime}=B^{\prime}=0$, then we get that both $n+d- \deg g$ and $n - \deg g$ are divisible by $p$, and thus $p \mid d$. Since $C$ can be written as $C=X^{n+1} r(X)$, where $r(X)$ is some polynomial with $r(0)\ne 0$, if $C^{\prime}=0$, then we must have $p \mid n+1$.

Thus, under the condition $p \nmid n+1$ or $p \nmid d$, the identity $A^{\prime}=B^{\prime}=C^{\prime}=0$ is not true, then $m=0$. So, the desired result follows from \eqref{eq:general deg} directly.
\end{proof}

We want to point out that the conclusions in Theorem \ref{thm:Degree} also hold for consecutive polynomial sequences of finite but sufficiently large length. One can understand other relevant results in this paper from the same point of view.

Now, we want to give an example to show that without the condition $p \nmid n+1$ or $p \nmid d$, the case  $A^{\prime}=B^{\prime}=C^{\prime}=0$ can happen in the proof of Theorem \ref{thm:Degree}.

\begin{example}
{\rm
 Choose $q=3, u_n=1$ for all integers $n\ge 0$, and use the notation in the proof of Theorem \ref{thm:Degree}. Fix $n=56$ and pick $d=3$, then we have
 \begin{align*}
 & f_{n} = (X^{54} + X^{51} + \cdots + X^3 +1)(X^2+X+1), \\
 & f_{n+d} = (X^{57} + X^{54} + \cdots + X^3 +1)(X^2+X+1).
 \end{align*}
So, we can get that $m=1$, and $\gcd(f_n,f_{n+d})=X^2+X+1$. It is easy to see that  $A^{\prime}=B^{\prime}=C^{\prime}=0$.
 }
\end{example}

Moreover, we can get the following asymptotic result.

\begin{corollary}
\label{cor:degree}
Let $\{ f_n \}$ be a consecutive polynomial sequence of infinite length.
For almost all integers $n\ge 1$, we have
$$
D(f_n) \gg \frac{\log n}{\log q}.
$$
\end{corollary}

 \begin{proof}
 By \eqref{eq:Degree1}, there exists an absolute constant $c$ such that
 \begin{equation}
 \label{eq:nnd}
 \max\{D(f_n),D(f_{n+d})\} \ge \frac{c\log (n+d)}{\log q}
 \end{equation}
 for any integer $n\ge 2q-1$ and any $0< d\le \frac{\log((n+1)/2)}{\log q}$ (note that the choice of $c$ is independent of $q$).

 Now, for any sufficiently large $n$, if $D(f_n) < \frac{c \log n}{\log q}$, then by \eqref{eq:nnd}, for any $0< d\le \frac{\log((n+1)/2)}{\log q}$, we have
 $$
 D(f_{n+d}) \ge \frac{c\log (n+d)}{\log q}.
 $$
 This implies  that
 $$
 \lim_{N\to \infty} \frac{|\{1\le n \le N: D(f_n) < \frac{c \log n}{\log q} \}|}{N} =0,
 $$
 which completes the proof.
 \end{proof}

We present another direct consequence of Theorem \ref{thm:Degree}.

\begin{corollary}
Let $\{ f_n \}$ be any consecutive polynomial sequence such that all the polynomials split completely over $\F_{q^k}$ for some fixed integer $k\ge 1$.
Then, the length of the sequence $\{f_n\}$ is less than $2q^k$.
\end{corollary}
\begin{proof}
Notice that  the largest degree of irreducible factors of the polynomials $f_n$ is at most $k$, then the desired result follows from \eqref{eq:Degree2}
(choosing $n=2q^k-1$ and $d=1$ there).
\end{proof}

Theorem \ref{thm:Degree} tells us that there exist irreducible factors with arbitrary large degree
in a given sequence $\{ f_n \}$ of infinite length. However,  it is generally false that $D(f_n)$ grows with $n$ or even
that $D(f_n)>1$ for all sufficiently large $n$.
As an example, by taking $u_n=1$ for all $n\ge 0$, it is easy to check
that
\begin{equation*}
  f_n(X)(X-1)= X^{n+1}-1, \quad n \ge 1 .
\end{equation*}
Fix an integer $n\ge 1$ and write $n+1=p^k m$ with $\gcd(m,p)=1$,
then according to \cite[Theorem 2.47]{LN}, $D(f_n)$ is exactly
the multiplicative order of $q$ modulo $m$.
 Especially, when $n+1=p^k$ for some integer $k$, then
$f_n(X)(X-1)= (X-1)^{p^k}$, and thus $D(f_n)=1$.

In addition, given two non-zero coprime integers $g,m$ with $m\ge 1$, denote by $\ell_g(m)$ the multiplicative
order of $g$ modulo $m$. In \cite[Theorem 1]{KP} (see \cite[Theorem 3.4]{Shp} for previous work), the authors have showed that if the Generalized Riemann Hypothesis is true, then for the average multiplicative order, we have
\begin{equation}
\label{eq:Kurlberg}
\frac{1}{x} \sum_{\substack{m\le x \\ \gcd(m,g)=1}} \ell_g(m) =
\frac{x}{\log x}
\exp \( \frac{B \log \log x}{\log \log \log x}(1 + o(1)) \)
\end{equation}
as $x\to \infty$ , uniformly in $g$ with $1 < |g| \le \log x$, where $B$ is an absolute constant defined by
$$
B = \exp(-\gamma) \prod_{\textrm{prime $k$}} \Big( 1 - \frac{1}{(k-1)^2(k+1)}\Big) = 0.345372 \ldots,
$$
where $\gamma$ is the Euler-Mascheroni constant.
 This can give a conditional asymptotic formula of the average value of $D(f_n)$ for the above sequence $\{f_n\}$.

\begin{theorem}
\label{thm:average dep}
Let $\{f_n\}$ be the consecutive polynomial sequence such that all the coefficients of $f_n$ for any $n\ge 1$ are equal to 1. Under the Generalized Riemann Hypothesis, we have
$$
\frac{1}{x} \sum_{n\le x } D(f_n) =
\frac{x}{\log x}
\exp \( \frac{B \log \log x}{\log \log \log x}(1 + o(1)) \),
$$
as $x \to \infty$, where $B$ is the constant in \eqref{eq:Kurlberg}, and the implied constant depends on $p$.
\end{theorem}

\begin{proof}
From the above discussions, for any $n\ge 1$, $D(f_n)=\ell_q(m)$ for some integer $m$, where $n+1=p^k m$ with $\gcd(m,p)=1$
and $p$ is the characteristic of $\Fq$.
So using \eqref{eq:Kurlberg}, for sufficiently large $x$ we have
\begin{align*}
& \frac{1}{x} \sum_{n\le x } D(f_n) \\
& = \frac{1}{x} \sum_{\substack{m\le x+1 \\ \gcd(m,q)=1}} \ell_q(m) + \frac{1}{x} \sum_{\substack{m\le (x+1)/p \\ \gcd(m,q)=1}} \ell_q(m) + \frac{1}{x} \sum_{\substack{m\le (x+1)/p^2 \\ \gcd(m,q)=1}} \ell_q(m) + \cdots \\
& = \( 1 + \frac{1}{p^2} + \frac{1}{p^4} + \cdots\) \frac{x}{\log x} \exp \( \frac{B \log \log x}{\log \log \log x}(1 + o(1)) \) \\
& = \frac{x}{\log x}
\exp \( \frac{B \log \log x}{\log \log \log x}(1 + o(1)) \),
\end{align*}
as $x \to \infty$, where the implied constant depends on $p$. This completes the proof.
\end{proof}

Theorem \ref{thm:average dep} suggests that the bound in Corollary \ref{cor:degree} might be not tight for the sequence $\{f_n\}$ in Theorem \ref{thm:average dep} and thus might be not optimal in general.

Furthermore, we can say more about the above sequence $\{f_n\}$. One can see that $f_n$ is irreducible if and only if $n+1$ is a prime number coprime to $q$ and $q$ is a primitive root modulo $n+1$. Recall that $q=p^s$. If $s$ is even, then $q$ is not a primitive root modulo $n+1$ whenever $n+1$ is an odd prime,
and thus $f_n$ is reducible for any $n\ge 2$. Otherwise, if $s$ is odd, under Artin's conjecture on primitive roots, there are infinitely many integers $n$ such that $f_n$ is irreducible.

\section{The Number of Irreducible Factors}
\label{sec:numberFactors}

Recall that for a polynomial $f\in \Fq[X]$,
 $\omega(f)$ stands for the number of  distinct monic irreducible factors of $f$.
In this section, we study irreducible factors of consecutive polynomial sequences.
First we need a lemma based on similar ideas as in~\cite[Lemma 1]{PShp1}.

\begin{lemma}
\label{lem:divisor}
Let $\{ f_n \}$ be any consecutive polynomial sequence of infinite length.
Given a non-constant polynomial $g \in \F_q[X]$, $g(0)\ne 0$, and  integers $m\ge 0$ and $H\ge 2$,
denote by $T(m,H;g)$ the number of positive integers $n$ with $m+1\le n \le m+H$ such that $g \mid f_n$,
and let $e(m,H;g)$ be the power of $g$ in the product $f_{m+1}f_{m+2}\cdots f_{m+H}$.
Then, we have
$$
T(m,H;g) \le 1 + H / \deg g,
$$
and
$$
e(m,H;g) \ll \frac{m+H\log (m+H)}{ \deg g}.
$$
In particular, if $H\ge 3$, we have
$$
T(0,H;g) \le H / \deg g \mand e(0,H;g) \le \frac{2H\log H}{ \deg g}.
$$
\end{lemma}

\begin{proof}
For any integers $n,d\ge 1$,  by construction in \eqref{eq:conspolyseq} we have
$$
f_{n+d}= f_n + X^{n+1} \sum_{i=1}^{d} u_{n+i}X^{i-1}.
$$
If $g\mid f_n$, then we can see that
$g\mid f_{n+d}$ if and only if $g \mid \sum_{i=1}^{d} u_{n+i}X^{i-1}$.
Thus, if $g\mid f_n$ and $d \le \deg g$, then we must have $g \nmid f_{n+d}$.
This implies that
$$
T(m,H;g) \le 1 + H / \deg g.
$$

Now, let $\theta(m,H;g)$ be the maximal power of $g$ in the factorizations of the
 polynomials $f_{m+1},f_{m+2},\ldots,f_{m+H}$. Then, we deduce that
\begin{align*}
e(m,H;g) = \sum_{k=1}^{\theta(m,H;g)} T(m,H;g^k)
& \le \sum_{k=1}^{\fl{(m+H)/\deg g}} (1 + H / (k\deg g))  \\
& \ll \frac{m+H\log (m+H)}{ \deg g}.
\end{align*}

For the case $m=0$, one can apply the same arguments to get the desired explicit estimates without
using the symbol ``$\ll$''. Here, in order to bound $e(0,H;g)$, one should use the assumption $H\ge 3$ and also the trivial upper bound for the partial sum of the harmonic series:
$$
\sum_{k=1}^{n} 1/k \le 1+ \log n, \quad n\ge 1.
$$
This completes the proof.
\end{proof}

Now, we are ready to estimate the number of distinct monic irreducible factors of the product of consecutive terms in a
consecutive polynomial sequence $\{f_n\}$, similarly as in~\cite[Theorem 2]{PShp1}.

\begin{theorem}
\label{thm:factorH}
Let $\{ f_n \}$ be any consecutive polynomial sequence of infinite length. For any integers $m\ge 0$ and $H\ge 2$, we have
$$
\omega \( f_{m+1}f_{m+2}\cdots f_{m+H} \) \gg \frac{(m+H)H}{m + H\log (m+H)}.
$$
In particular, if $H\ge 3$, we have
$$
\omega \( f_1f_2\cdots f_H \) \ge H/ (4\log H).
$$
\end{theorem}

\begin{proof}
It follows from Lemma \ref{lem:divisor} that for any irreducible polynomial $g\in \F_q[X]$ we get
$$
\deg (g^{e(m,H;g)}) \ll m + H\log (m+H).
$$
On the other hand, since $\deg f_n =n$ for any $n\ge 1$, we have
$$
\deg (f_{m+1}f_{m+2}\cdots f_{m+H})\gg mH+H^2.
$$
Thus, the above two bounds yield the first desired result.

The second desired lower bound can be obtained by applying the same arguments and using the
explicit estimates in Lemma \ref{lem:divisor}.
\end{proof}

Let $\cS$ be a finite set of irreducible polynomials in $\F_q[X]$.
We call a polynomial $f\in \F_q[X]$ an $\cS$-polynomial if all its irreducible factors are
contained in $\cS$.

\begin{theorem}
Let $\{ f_n \}$ be a consecutive polynomial sequence of infinite length.
For any integers $m\ge 0$ and $H\ge 2$, denote by $Q(m,H;\cS)$ the number of $\cS$-polynomials amongst
$f_{m+1},f_{m+2},\ldots,f_{m+H}$. Then, we have
$$
Q(m,H;\cS) \ll |\cS|\log H\log(m+H).
$$
\end{theorem}

\begin{proof}
We follow that same approach as in~\cite[Theorem 3]{PShp1}.

Set $L_0=1$.
Split the interval $[1, H]$ into $k=O(\log H)$ intervals $[L_{i-1},L_i]$,
where $L_i=\min \{2^i, H \}$, $i=1,2,\ldots,k$. For any $1\le i \le k$, let $M_i$ be the number of $\cS$-polynomials
among $f_n, n\in [m+L_{i-1},m+L_i]$. Since $\deg f_n =n$ for each $n\ge 1$
and $L_i \le 2L_{i-1}$ for any $1\le i \le k$, combining with Lemma \ref{lem:divisor}, we obtain
\begin{align*}
(m+L_{i-1})M_i
& \le \sum_{g \in \cS} \deg (g^{e(m+L_{i-1}-1,L_{i-1}+1;g)}) \\
& \ll |\cS| (m+L_{i-1}) \log (m+L_i).
\end{align*}
So, we get $M_i \ll |\cS| \log (m+L_i)$. Thus,
$$
Q(m,H;\cS) = \sum_{i=1}^{k} M_i \ll |\cS|\sum_{i=1}^{k} \log (m+L_i) \ll |\cS|\log H\log(m+H).
$$
This completes the proof.
\end{proof}

The lower bound in \eqref{eq:Lq lower} says that when $q$ is large enough,
there exists a consecutive irreducible polynomial sequence whose length is greater than $\frac{\log q}{2\log\log q}$.
We can improve this lower bound if we want to search for a consecutive polynomial sequence
whose terms are pairwise coprime.

\begin{theorem}
\label{thm:coprime}
There exists a consecutive polynomial sequence $\{ f_n \}$ over $\F_q$ of length
$$
H \ge \lfloor \sqrt{2(q-1)} + 3/2 \rfloor
$$
such that all the terms in the sequence are pairwise coprime.
\end{theorem}

\begin{proof}
First, we note that two polynomials $f$ and $g$ are coprime if and only if
their resultant $\Res{f,g}\ne 0$.
So, given a consecutive polynomial sequence $\{f_n\}$ of length $H \ge 3$ defined by \eqref{eq:conspolyseq} such that $f_1,f_2,\ldots,f_{H-1}$ are pairwise coprime,
the polynomials $f_1,f_2,\ldots,f_H$ are pairwise coprime if and only if
\begin{equation}
\label{eq:res}
\prod_{1\le i \le H-2} \Res{f_i,f_H} \ne 0,
\end{equation}
where one should note that $f_H$ and $f_{H-1}$ are automatically coprime.

Note that for each $1\le i \le H-2$, $\Res{f_i,f_H}$ is a polynomial in $u_H$ of degree at most $i$,
and thus the polynomial $\prod_{1\le i \le H-2} \Res{f_i,f_H}$ has at most $(H-1)(H-2)/2$ zeros.
So, when $q-1 > (H-1)(H-2)/2$, we can choose non-zero $u_H \in \Fq$ such that the inequality \eqref{eq:res} holds;
that is, we get a consecutive polynomial sequence of length $H$ whose terms are pairwise coprime.
Hence, we need to ensure that $(H-3/2)^2 < 2q - 7/4$, for which it suffices to choose
$$
H = \lfloor \sqrt{2(q-1)} + 3/2 \rfloor.
$$
This completes the proof.
\end{proof}

\begin{example}
{\rm
In Table \ref{tab:L(q)}, we can see that the maximum length of consecutive irreducible polynomial sequences over $\F_3$ is equal to $3$.
It is easy to check that the following consecutive sequence of polynomials over $\F_3$ has pairwise coprime terms:
\begin{align*}
& f_1 = X+1, \quad f_2 = 2X^2+X+1, \quad f_3= X^3 + 2X^2 +X +1, \\
&  f_4= X^4+X^3 + 2X^2 +X +1, \quad  f_5= X^5 + X^4+X^3 + 2X^2 +X +1, \\
&  f_6= X^6+X^5+X^4+X^3 + 2X^2 +X +1.
\end{align*}
Here, both $f_4$ and $f_6$ are reducible polynomials. In fact, we have
$$
f_4  = (X-1)^2(X^2+1), \quad f_6 = (X^3+2X+1)(X^3+X^2+2X+1).
$$
 }
\end{example}

\section{The Number of Consecutive Irreducible polynomial Sequences}
\label{sec:numberconsecirreduc}

 Recall that for any integer $N\ge 2$, $I_N$ is the number of consecutive irreducible
polynomial sequences of length $N$. In this section, we give some upper and lower bounds for $I_N$,
as well as an asymptotic formula.

\subsection{Trivial bound}

For an integer $n\ge 1$,  let $\pi_q(n)$ be the number of monic  irreducible polynomials of degree $n$
 over $\Fq$.
By \cite[Lemma 4]{Pollack}, we have
\begin{equation}
\label{eq:PNT}
\frac{q^n}{2n} \le \pi_q(n) \le \frac{q^n}{n}.
\end{equation}

Trivially, we have that $I_N$ is not greater than the number of irreducible polynomials of degree $N$ over $\Fq$.
So for $N\ge 2$, by \eqref{eq:PNT} we have
\begin{equation}
\label{eq:trivial}
I_N \le (q-1) \frac{q^{N}}{N} < \frac{q^{N+1}}{N}.
\end{equation}

\subsection{Nontrivial upper bounds}
\label{sec:upper}
Here, under some circumstances we establish some upper bounds for $I_N$ better than the trivial one in \eqref{eq:trivial}.

\begin{theorem}
\label{thm:upper_bound1}
For any integer $N\ge 2$, the number $I_{N}$ of consecutive irreducible polynomial sequences of length $N$ satisfies
$$
I_N < 3^{-N/7+1} q^{N+1}.
$$
\end{theorem}

\begin{proof}
By \eqref{eq:trivial}, we have $I_N < q^{N+1} / N$ for any $N \ge 2$.
 It is easy to check that this is better than the desired upper bound for $I_N$ when $2 \le N \le 7$.
Now, assume that $N\ge 8$.

Note that for each consecutive irreducible polynomial sequence $\{ f_n \}$ of length $N$ defined by \eqref{eq:conspolyseq} and for any positive integer $4 \le m\le N$, we have
$$
u_NX^N+ \cdots + u_1X + u_0 \equiv u_{N-m}X^{N-m} + \cdots + u_1X + u_0 \mod X^{N-m+1},
$$
which, together with Lemma \ref{lem:Johnsen}, implies that
\begin{equation}
\label{eq:INm}
I_N < I_{N-m} \cdot \frac{2q^{m}}{m-1}.
\end{equation}
Write $N=km + r$ with $0\le r < m$. Using \eqref{eq:INm} repeatedly, we obtain
$$
I_N < I_r  \( \frac{2q^{m}}{m-1} \)^k,
$$
where one should note that $I_0=q-1$ and $I_1= (q-1)^2$.

Applying the trivial estimate $I_r < q^{r+1}$, we have
\begin{equation}
\label{eq:IN-upper}
I_N < q^{N+1} \( \frac{2}{m-1} \)^k \le q^{N+1} \( \frac{2}{m-1} \)^{N/m-1}
\end{equation}
for any $4 \le m \le N$.  Let
$$
g(m) = \log \( \frac{2}{m-1} \)^{N/m-1} = -\frac{N}{m} \log\frac{m-1}{2} + \log\frac{m-1}{2} .
$$
Then, to get a good upper bound for $I_N$, we need to compute the minimum value of $g(m)$ for integers $m$ with $4 \le m \le N$.
That is, we need to compute the maximum value of
$$
h(m)=\frac{1}{m} \log\frac{m-1}{2}, \quad  4 \le m \le N.
$$
It is easy to see that the function $h(m)$ attains its maximum value at $m=7$ when $m$ runs through all the integers not less than $4$.
Since we have assumed that $N\ge 8$, we can achieve this maximum value.
Hence, in \eqref{eq:IN-upper} we choose $m=7$ for $N\ge 8$. This completes the proof.
\end{proof}

In the following, we want to improve the upper bound in Theorem \ref{thm:upper_bound1} when $q$ is much larger than $N$.
To give such an improvement on bounding $I_N$, we use the same technique as in~\cite[Theorem 5.5]{GNOS}. For this we need the following lemma.

\begin{lemma}
\label{lem:sqfree}
Let $\{ f_n \}$ be a consecutive irreducible polynomial sequence
defined in~\eqref{eq:conspolyseq}.
Then, for any $\nu\ge 1$,
$$
D_{n_1,\ldots,n_{\nu}}=\prod_{j=1}^{\nu} \Disc{f_{n_j}},\quad 2\le n_1<\ldots<n_{\nu},
$$
is not a square polynomial in  $u_0,u_1,\ldots,u_{n_\nu}$ (as a multivariate polynomial).
\end{lemma}

\begin{proof}
The proof follows by induction on $\nu\ge 1$. Although the sequence $\{u_n\}$ is given in \eqref{eq:conspolyseq},
we sometimes view $u_0,u_1,\ldots$ as variables when considering discriminants without specific indication.

For the induction argument we need to prove that $D_{n_1}$ and  $D_{n_1,n_2}$ are not square polynomials.

We prove first that $D_{n_1}$ is not a square polynomial. If  $D_{n_1}=\Disc{f_{n_1}}$ were a square polynomial as a multivariate polynomial in $u_0,\ldots,u_{n_1}$, then
for any specialisation of the variables $u_0,\ldots,u_{n_1}$, we would get that  $\Disc{f_{n_1}}$ is a square element
in $\Fq$. From Lemma \ref{lem:Stick}, this implies that for any choice of $u_0,\ldots,u_{n_1}\in\Fq$, $u_{n_1}\ne
0$, the number of irreducible factors of $f_{n_1}$ is congruent to $n_1$
modulo $2$ when $f_{n_1}$ is square-free, which is obviously not true in general. Thus, $D_{n_1}$ is not a square multivariate polynomial.

We prove now that $D_{n_1,n_2}$ is not a square polynomial in $u_0,\ldots,u_{n_2}$. If $D_{n_1,n_{2}}$ is a square polynomial, then it is also a square polynomial for the specialisation $u_{n_2}=0$.
Using \eqref{eq:disc n^n-1} with $u_{n_2}=0$, we get
$$
\Disc{f_{n_2}} \Big\rvert_{u_{n_2}=0}=u_{n_2-1}^2\Disc{f_{n_2-1}},
$$
which implies that
$$
D_{n_1,n_2}\Big\rvert_{u_{n_2}=0} =u_{n_2-1}^2\Disc{f_{n_2-1}}\Disc{f_{n_1}},
$$
which is a square if and only if $\Disc{f_{n_1}}\Disc{f_{n_2-1}}$ is a square.

If $n_2-1>n_{1}$ we continue the same process as above, that is, if $\Disc{f_{n_1}}\Disc{f_{n_2-1}}$ is a square then it is a square for the specialisation $u_{n_2-1}=0$. From~\eqref{eq:disc n^n-1}, we get
$$
\Disc{f_{n_2-1}}\Big\rvert_{u_{n_2-1}=0}=u_{n_2-2}^2\Disc{f_{n_2-2}}.
$$
We apply this reduction until we obtain $n_{2}-k=n_{1}+1$, that is for $k=n_{2}-n_{1}-1$ times. Putting everything together we get that if $D_{n_1,n_2}$ is a square polynomial, then so is
$$
\(\prod_{k=1}^{n_{2}-n_{1}-1}u_{n_2-k}\)^2\Disc{f_{n_1}}\Disc{f_{n_1+1}},$$
and thus $\Disc{f_{n_1}}\Disc{f_{n_1+1}}$ is also a square polynomial. Using~\eqref{eq:discprod}, this is equivalent with that $\Disc{f_{n_1}f_{n_1+1}}$ is a square polynomial. Suppose that $\Disc{f_{n_1}f_{n_1+1}}$ is a square polynomial in $u_0,\ldots,u_{n_1+1}$, then by Lemma \ref{lem:Stick}, the number of irreducible factors of $f_{n_1}f_{n_1+1}$, which is exactly $2$ (as $f_{n_1}$ and $f_{n_1+1}$ are irreducible), is congruent to $1$ modulo $2$ (as $2n_1+1$ is the degree of $f_{n_1}f_{n_1+1}$); this is not true. We finally conclude that $D_{n_1,n_2}$ is not a square polynomial.

We now assume that $\nu \ge 3$ and the statement is true for $D_{n_1,\ldots,n_{j}}$ for any $j\le \nu-1$.
If $D_{n_1,\ldots,n_{\nu}}$ is a square polynomial, then using exactly the same reductions as the above (using~\eqref{eq:disc n^n-1}), but $n_{\nu}-n_{\nu-1}$ times, we obtain that
\begin{equation*}
\begin{split}
\(\prod_{k=1}^{n_{\nu}-n_{\nu-1}}u_{n_\nu-k}\)^2&\Disc{f_{n_{\nu-1}}}D_{n_1,\ldots,n_{\nu-1}}\\
&=\(\Disc{f_{n_{\nu-1}}}\prod_{k=1}^{n_{\nu}-n_{\nu-1}}u_{n_\nu-k}\)^2D_{n_1,\ldots,n_{\nu-2}}
\end{split}
\end{equation*}
is also a square polynomial.
Thus, $D_{n_1,\ldots,n_{\nu-2}}$ is a square polynomial, which contradicts the induction hypothesis.
 Now, we conclude the proof.
\end{proof}

\begin{remark}
{\rm
In the third paragraph of the above proof, we actually prove that for any polynomial $f\in \F_q[X]$ of degree greater than 1, its discriminant is not a square polynomial as a multivariate polynomial in the coefficients of $f$ (treated as variables).
}
\end{remark}

Now, we are ready to get a better upper bound for $I_N$  when $q$ is very large compared to $N$.

\begin{theorem}
\label{thm:upper_bound}
For any integer $N\ge 2$, the number $I_{N}$ of consecutive irreducible polynomial sequences of length $N$ satisfies
$$
I_N < 2^{-N+1}q^{N+1}+N^2 q^{N+1/2}+ N^4 q^{N}.
$$
\end{theorem}

\begin{proof}
Let $\{ f_n \}$ be a consecutive polynomial sequence of length $N$  defined in \eqref{eq:conspolyseq}.
 If $f_2,\ldots,f_N$
are irreducible polynomials,  by Lemma~\ref{lem:Stick}, we know that
\begin{equation*}
  \chi\(\Disc{f_n}\)=(-1)^{n+1}, \quad n=2,3,\ldots,N,
\end{equation*}
where $\chi$ is the multiplicative quadratic character of $\Fq$.
By convention, we put $\chi(0)=0$.

Thus, we have
\begin{equation}
\label{eq:IN}
\begin{split}
I_N & \le \sum_{u_0,\ldots,u_N\in\Fq}\frac{1}{2^{N-1}}\prod_{n=2}^N\(1-(-1)^n\chi(\Disc{f_n})\) \\
& = \frac{1}{2^{N-1}}\sum_{u_0,\ldots,u_N\in\Fq}\prod_{n=2}^N\(1-(-1)^n\chi(\Disc{f_n})\).
\end{split}
\end{equation}

Just expanding the product in~\eqref{eq:IN}, we obtain $2^{N-1}-1$  character sums
 of the shape
\begin{equation}
\label{eq:Sums}
(-1)^{\nu+n_1+\dots+n_\nu}q^{N-n_\nu}\sum_{u_0,\ldots,u_{n_\nu}\in\Fq} \chi\(\prod_{j=1}^\nu \Disc{f_{n_j}}\),
\end{equation}
where $2 \le n_1 < \cdots < n_\nu\le N$, and one trivial sum that equals $q^{N+1}$ (corresponding to
the terms $1$ in the product of~\eqref{eq:IN}).

So, the trivial summand of the right-hand side in~\eqref{eq:IN} is equal to $q^{N+1}/2^{N-1}$.
We view each $\prod_{j=1}^\nu \Disc{f_{n_j}}$ as a multivariate polynomial in $u_0,u_1,\ldots,u_{n_\nu}$,
whose degree is equal to
$$
\sum_{j=1}^{\nu}(2n_j - 2) = 2(n_1+\cdots +n_\nu) - 2\nu \le N^2.
$$

Note that if we associate values to $n_\nu $ variables among $u_0,u_1,\ldots,u_{n_\nu}$, the resulted polynomial
might be a square polynomial in the remaining variable up to a multiplicative constant.
By Lemma~\ref{lem:sqfree}, we know that $\prod_{j=1}^\nu \Disc{f_{n_j}}$ is not a square polynomial  in $u_0,u_1,\ldots,u_{n_\nu}$, and thus by Lemma~\ref{lem:square} we obtain that
there exists $i \in \{0,1, \ldots, n_\nu\}$ such that
$\prod_{j=1}^\nu \Disc{f_{n_j}}$
is not a square polynomial in $u_i$ up to a multiplicative constant for all but at most
$N^4q^{{n_\nu}-1}$  values  of
$$
u_0,\ldots,u_{i-1}, u_{i+1},\ldots,u_{n_\nu} \in \F_q.
$$

We use Lemma~\ref{lem:Weil} for those specialisations for which $\prod_{j=1}^\nu \Disc{f_{n_j}}$
is not a square polynomial in $u_i$ up to a constant; and for the rest, we use the trivial bound.
Thus, we deduce that
\begin{equation*}
\begin{split}
\left | q^{N-n_\nu} \sum_{u_0,\ldots,u_{n_\nu}\in\Fq} \chi\(\prod_{j=1}^\nu \Disc{f_{n_j}}\) \right |
&  \le q^{N-n_\nu}(N^2q^{n_\nu+1/2} + N^4q^{n_\nu}) \\
&  \le N^2q^{N+1/2} + N^4q^{N}.
 \end{split}
\end{equation*}

So, regarding \eqref{eq:IN} and putting everything together, we obtain
\begin{equation*}
\begin{split}
I_N  & \le q^{N+1}/2^{N-1}+ \frac{2^{N-1}-1}{2^{N-1}}\( N^2 q^{N+1/2}+ N^4 q^{N} \) \\
& < 2^{-N+1}q^{N+1}+  N^2 q^{N+1/2}+ N^4 q^{N},
 \end{split}
\end{equation*}
which completes the proof.
\end{proof}

We remark that when $N\ge 4$ and $q\ge 3^{2N/7}N^4$,
for the three summation terms in the bound of Theorem \ref{thm:upper_bound} each of them is not greater than one third of the bound $3^{-N/7+1}q^{N+1}$ in Theorem \ref{thm:upper_bound1},
so Theorem \ref{thm:upper_bound} is better than the ibound in Theorem \ref{thm:upper_bound1}.

\subsection{Heuristic approximation}
\label{sec:heuristic}

Here, we present a heuristic estimate for $I_N, N\ge 2$, which is compatible with numerical data and implies an upper bound for $L(q)$ (defined in Section \ref{sec:motivation}).

Heuristically, from each polynomial $g(X)$ contributing to $I_{N-1}$, we seek through $q-1$ values of $u_N \in \F_q^*$ such that $u_N X^N + g(X)$ is irreducible.
For $N \ge 2$, a naive approximation to the number of irreducible polynomials $f(X)\in \F_q[X]$ satisfying
$$
\deg f = N \mand  f \equiv g \pmod{X^N}
$$
is
\begin{equation}
\label{eq:recurrence}
\frac{(q-1)q^N}{q^{N-1}(q-1)N} = q/N;
\end{equation}
see \cite[Theorem 4.8]{Rosen}.
However, since we require that $u_N \ne 0$, we need to introduce a correction factor $(q-1)/q$. Thus, we are led to the
following approximate recurrence relation
$$
I_N \approx \frac{q-1}{N} I_{N-1},
$$
which, together with the initial value $I_1 = (q-1)^2$, implies the approximation
\begin{equation}
\label{eq:heuristic}
I_N \approx \frac{(q-1)^{N+1}}{N!}.
\end{equation}

Figure~\ref{fig:imagen} illustrates the comparison between the number of consecutive irreducible polynomial sequences and the approximation \eqref{eq:heuristic} for $q=17$, where the horizontal axis represents $N$. From Figure~\ref{fig:imagen} one can see that  \eqref{eq:heuristic} approximates $I_N$ very well.

 \begin{figure}
  %\centering
  \includegraphics[scale=0.4]{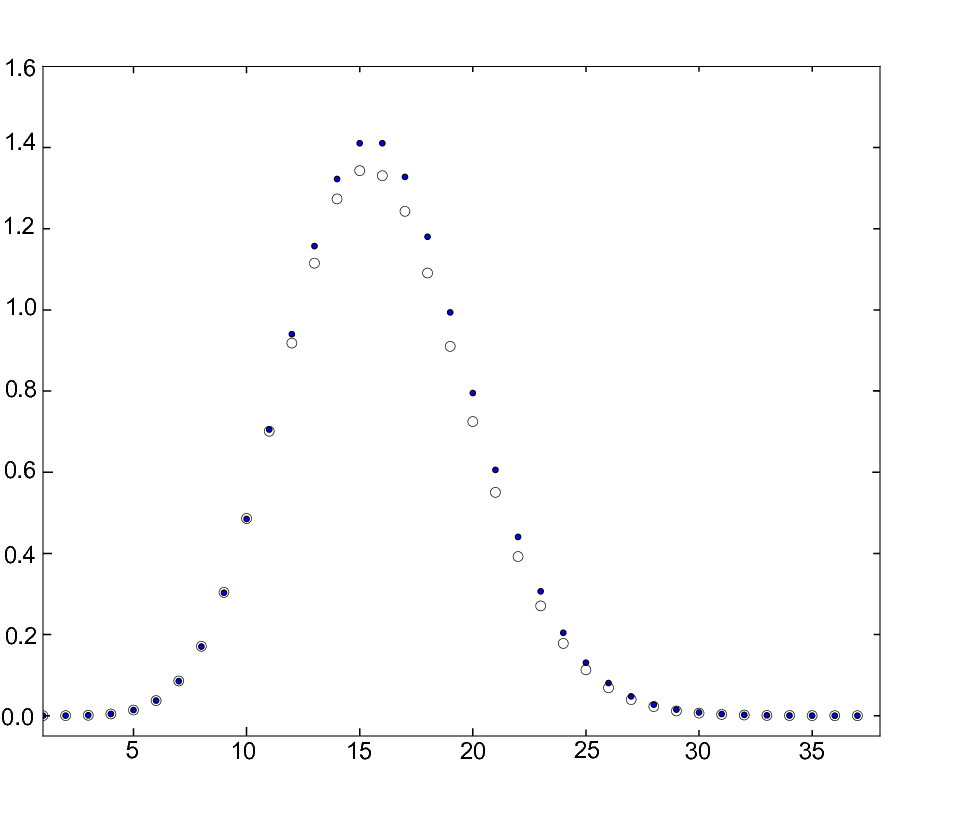}
  \caption{Comparison between $I_N$ (circles) and  the approximation in \eqref{eq:heuristic} (dots) for $q=17$. The $x$-axis represents the value of $N$
  and the scale of the $y$-axis is multiplied by $10^7$.}
  \label{fig:imagen}
\end{figure}

Now, as suggested by \eqref{eq:PNT} and Figure~\ref{fig:imagen}, we view the approximation in \eqref{eq:heuristic} as an upper bound of $I_N$.
Using the standard estimate on the factorials (for example, see \cite{Robbins}):
$$
N! > \sqrt{2\pi N}\( N/ e\)^N,
$$
where $e$ is the base of the natural logarithm,
we obtain
$$
(q-1)^{N+1}/N! < 1
$$
 when $N\ge 3q$. Thus, under the heuristic upper bound suggested in  \eqref{eq:heuristic}
 we have $I_N = 0$ for $N\ge 3q$, and so  $L(q)< 3q$,
 where $L(q)$ is the maximal length of consecutive irreducible polynomial sequences.
 Besides, note that $I_{L(q)+1}$ = 0 by the definition of $L(q)$, so in \eqref{eq:recurrence} we heuristically have $q/(L(q)+1) <1$,
 which implies that $q \le L(q)$.
 Thus, we obtain the following heuristic estimate
 \begin{equation}
 \label{eq:L(q)}
 q \le L(q)< 3q,
 \end{equation}
  which is compatible with Table~\ref{tab:L(q)}.

Hence, heuristically there is no consecutive irreducible polynomial sequence of infinite length over $\Fq$.

\begin{table}
\centering
\caption{Values of $L(q)$ for small $q$}
\label{tab:L(q)}
\begin{tabular}{|c|c|c|c|c|c|c|c|c|c|c|c|}
\hline
$q$ & 3 & 5 & 7 & 9 & 11 & 13 & 17 & 19 & 23  \\ \hline

$L(q)$ & 3 & 6  &  8 & 16 & 23 & 29 & 38 & 41 & 47  \\ \hline

\end{tabular}
\end{table}

\subsection{Lower bounds}

In the proof of \cite[Theorem 1.2]{CC} (noticing Theorem 1.1 and Equation (3.5) there), the authors actually have given an asymptotic formula for $I_N$ with respect to $q$:
\begin{equation}
\label{eq:asymp}
I_N = \frac{q^{N+1}}{N!} + O_N \( q^{N+1/2} \),
\end{equation}
where, in particular, the implied constant can be easily computed for small $N$.

The approach in \cite{CC}, using reciprocal polynomials, is first fixing a consecutive irreducible sequence $u_0,\ldots,u_{n-1},$
and then searching $u_n$ such that the polynomial $(u_0X^{n-1} + \cdots + u_{n-1})X + u_n$ is irreducible.
In this section, our approach is to searching $u_n$ such that the polynomial $u_0 + \cdots + u_{n-1}X^{n-1} + u_nX^n$ is irreducible.
This enables us to obtain new explicit lower bounds for $I_N$ when $N$ is small,
which are better than those implied in \cite{CC}.

We first remark that the number of consecutive irreducible polynomial sequences of fixed length
is divisible by $(q-1)^2$. Indeed, let $\{ f_n \}$ be a consecutive polynomial sequence
defined by a sequence  $\{ u_n \}$ of $\Fq$ as in~\eqref{eq:conspolyseq}. Then, we know that for any $n\ge 1$ and $a\in\Fq^*$, $f_n$ is irreducible if
and only if $f_n(aX)$ or $af_n(X)$ is irreducible. Thus, $\{ u_n \}$ is a consecutive
irreducible sequence if and only if $\{a^nu_n\}$ or $\{au_n\}$
is a consecutive irreducible sequence. In particular, when $\{ u_n \}$ is a consecutive
irreducible sequence, all these $(q-1)^2$ consecutive sequences $\{ab^nu_n\}$, where $a,b$
run over $\F_q^*$, are irreducible and pairwise distinct.

Next, we give some estimates for such polynomial sequences of length $2,3$ and  $4$,  which are compatible with \eqref{eq:asymp}.

\begin{theorem}
\label{thm:I2I3}
The following  hold:

{\rm (1)}   $I_2 = \frac{1}{2}(q-1)^3$;

{\rm (2)}  $I_3 \ge \frac{1}{6}(q-1)^2(q-9)(q-2\sqrt{q}-6)$ when $q > 13$.
\end{theorem}

\begin{proof}
{\bf (1)}
  As we have noted at the beginning of this section, if $\{ u_n \}$ is a
  consecutive irreducible sequence of $\Fq$, $\{ ab^nu_n \}$ are all distinct and
  consecutive irreducible sequences when $a,b$ run through $\F_q^{*}$.

  Therefore, we fix $u_0=u_1=1$.
  By Lemma \ref{lem:Stick}, a quadratic polynomial $u_2 X^2+X+1$ is irreducible if and only if its discriminant
  is not a square element in $\F_q$.
 Since the discriminant is $1-4u_2$,  we have
  \begin{equation}
  \label{eq:I2}
    \frac{I_2}{(q-1)^2}=\#\{u_2\in\F_q^*\ \mid\ 1-4u_2 \text{ is not a
      square in }\F_q \},
  \end{equation}
  which, by noticing that there are exactly $(q-1)/2$ non-square elements in $\F_q$,
  in fact is equal to $(q-1)/2$.
  This gives us the desired result.

{\bf (2)}
  For $I_3$,  first fix $u_2$ such that the polynomial $u_2 X^2+X+1$ is irreducible, and then we proceed to
  count how many of the polynomials
  \begin{equation*}
    f_3=u_3X^3+u_2X^2+X+1,
  \end{equation*}
  are irreducible when $u_3$ runs over $\F_q^*$. The first thing to notice is that if $u_3\neq u_3'$,
  then
  \begin{align}
    \label{eq:gcd2}
    & \gcd(u_3X^3+u_2X^2+X+1,u_3'X^3+u_2X^2+X+1)\\
    & =\gcd(u_3X^3+u_2X^2+X+1,(u_3'-u_3)X^3)=1, \notag
  \end{align}
  which means that these two  polynomials have different irreducible
  factors.

In the following, without loss of generality we assume that the characteristic of $\F_q$ is not equal to $3$
(note that we have already assumed that this characteristic is not equal to $2$).
In fact, when the characteristic is equal to $3$, the situations in \eqref{eq:Disc f3} and \eqref{eq:Disc g} become simpler,
and so a better bound can be obtained by following similar arguments.

  By a simple calculation, the discriminant of $f_3$ is equal to
  \begin{equation}   \label{eq:Disc f3}
    \Disc{f_3}= -27 u_{3}^2 + (18 u_2 - 4) u_{3} - 4 u_{2}^{3} + u_{2}^{2}.
  \end{equation}
  Notice that this discriminant can be viewed as a quadratic polynomial in $u_3$ (because the characteristic of $\F_q$ is not equal to $3$),
  and it has no multiple roots if and only if its discriminant
  \begin{equation}
    \label{eq:condition_on_u_2}
    -432 u_{2}^3 + 432 u_{2}^2 - 144 u_{2} +  16 \neq 0.
  \end{equation}
  Let $\chi$ be the multiplicative quadratic character of $\F_q$.
  Now, under the condition \eqref{eq:condition_on_u_2}, which means that $\Disc{f_3}$ is not a square polynomial in $u_3$ up to a multiplicative constant, we estimate the number of $u_3$ such that $\Disc{f_3}$ is a square element in the following way:
  \begin{align*}
   \#\{u_3\in\F_q^*\;|\;\chi(\Disc{f_3})=1\}
   & \ge \#\{u_3\in\F_q\;|\;\chi(\Disc{f_3})=1\} -1 \\
   &\ge \left |\frac{1}{2}\sum_{u_3\in\F_q}(1+\chi(\Disc{f_3}))\right | - 2 \\
   &\ge \frac{q}{2} - \frac{1}{2}\left|\sum_{u_3\in\F_q}\chi(\Disc{f_3})\right| -2 \\
   & \ge q/2-\sqrt{q}/2 -2,
  \end{align*}
  where the second inequality comes from the two possible values of $u_3$ such that $\Disc{f_3} = 0$, and
   the last inequality comes from Lemma~\ref{lem:Weil}. Thus, using Lemma \ref{lem:Stick}, we get that under the condition \eqref{eq:condition_on_u_2},
  for at least
  \begin{equation} \label{eq:f3-irre}
  q/2-\sqrt{q}/2-2
  \end{equation}
  values of $u_3$ the polynomial $f_3$ has an odd number of distinct irreducible factors.

  If the polynomial $f_3$ is reducible and has an odd number of distinct irreducible factors, it must have three distinct roots in $\Fq^*$. By \eqref{eq:gcd2}, there are at most $(q-1)/3$ such polynomials $f_3$. But here we can get a better estimate. Assume that $\al_1, \al_2, \al_3 \in \Fq^*$ are three distinct roots of $f_3$, namely
  $$
  f_3 = u_3X^3 + u_2X^2 + X + 1 = u_3(X- \al_1) (X- \al_2) (X- \al_3).
  $$
  Then, we get
     \begin{equation}
  \label{eq:al123}
\left\{\begin{array}{ll}
  \al_1 + \al_2 + \al_3 =-u_2/u_3,\\
  \al_1\al_2 + \al_1\al_3 + \al_2\al_3 = 1/ u_3,\\
  \al_1\al_2\al_3 =-1/u_3.
\end{array}\right.
\end{equation}
Put $\be_i = \al_i^{-1}, i=1,2,3$. By \eqref{eq:al123}, we have
     \begin{equation}
  \label{eq:be123}
\left\{\begin{array}{ll}
  \be_1 + \be_2 + \be_3 =-1,\\
  \be_1\be_2 + \be_1\be_3 + \be_2\be_3 = u_2.
\end{array}\right.
\end{equation}
Note that $u_2$ is fixed. If we fix $\be_1$ (that is, $\al_1$), then $\be_2$ and $\be_3$ (that is, $\al_2$ and $\al_3$) are uniquely determined by \eqref{eq:be123},
and then $u_3$ is fixed.
Hence, it suffices to estimate the number of $\be_1 \in \Fq^*$ such that both $\be_2$ and $\be_3$ are in $\Fq^*$.
From \eqref{eq:be123}, $\be_2$ and $\be_3$ are the two distinct roots of the polynomial
$$
g = X^2 + (\beta_1 + 1)X + \beta_1^2 + \beta_1 + u_2.
$$
Note that the discriminant of $g$ is
\begin{equation}   \label{eq:Disc g}
\Disc{g} = -3 \beta_1^2 - 2 \beta_1 - 4u_2 + 1 \ne 0,
\end{equation}
which can be viewed as a quadratic polynomial in $\beta_1$ (because the characteristic of $\F_q$ is not equal to $3$).
By Lemma \ref{lem:Stick}, both roots of $g$ are in $\Fq$ if and only if $\Disc{g}$ is a square element in $\Fq$.

Now, we view $\Disc{g}$ as a polynomial in $\beta_1$.
Its discriminant is $16-48u_2$.
So, if $u_2 \ne 1/3$, $\Disc{g}$ has two distinct roots and thus is a square-free polynomial in $\beta_1$.
Then, using Lemma \ref{lem:Weil}, the number of $\beta_1 \in \Fq^*$ such that the roots $\beta_2,\beta_3$ of $g$ are in $\Fq$ can be  estimated as:
 \begin{align*}
   \#\{\beta_1\in\F_q^*\;|\;\chi(\Disc{g})=1\}
   & \le \#\{\beta_1\in\F_q\;|\;\chi(\Disc{g})=1\}  \\
   &= \left |\frac{1}{2}\sum_{\beta_1\in\F_q}(1+\chi(\Disc{g}))\right |  \\
   &\le \frac{q}{2} + \frac{1}{2}\left|\sum_{\beta_1\in\F_q}\chi(\Disc{g})\right|  \\
   & \le q/2 + \sqrt{q}/2.
  \end{align*}
 That is, there are at most $q/2 + \sqrt{q}/2$ such values of $\beta_1$.
 Hence, there are at most
 $$
 \frac{1}{3} (q/2 + \sqrt{q}/2) = q/6 + \sqrt{q}/6
 $$
possible cases of $f_3$ having three distinct roots in $\Fq$ when $u_2$ is fixed and $u_2 \ne 1/3$.
Thus, combining with \eqref{eq:f3-irre}, at least
  \begin{equation}  \label{eq:u3-irre}
  q/2-\sqrt{q}/2 -2 - (q/6 + \sqrt{q}/6) = q/3-2\sqrt{q}/3-2
  \end{equation}
   values of $u_3$ give an irreducible polynomial $f_3$ if $u_2$ satisfies \eqref{eq:condition_on_u_2} and $u_2 \ne 1/3$.

  In view of \eqref{eq:I2} and \eqref{eq:condition_on_u_2}, there are at least
  $$
  I_2/(q-1)^2-3-1 =   I_2/(q-1)^2- 4
  $$
  choices of $u_2$
  such that $u_2 \ne 1/3$, the polynomial $u_2X^2+X+1$ is irreducible and the condition \eqref{eq:condition_on_u_2} is satisfied.
  Thus, combining with \eqref{eq:u3-irre} we deduce that
  \begin{equation}
  \label{eq:I3}
  I_3 \ge (I_2/(q-1)^2-4)(q/3-2\sqrt{q}/3-2)(q-1)^2,
  \end{equation}
  which, together with the first result (1) of this theorem, implies the desired result.
 Note that, to ensure $q/3-2\sqrt{q}/3-2 > 0$, we need $q > 13$.
\end{proof}

The strategy to estimate $I_4$ is the same as in the proof of Theorem \ref{thm:I2I3},
but the deductions are much more complicated.

\begin{theorem}
\label{thm:I4}
Assume that the characteristic of $\Fq$ is not equal to $2$ or $3$. Then, for $q \ge 504$ we have
$$
  I_4 \ge \frac{1}{24} (q-1)^2 (q - 22 \sqrt{q} - 10) (q^2 - 2q\sqrt{q} - 18q + 18 \sqrt{q} +57).
$$
\end{theorem}

\begin{proof}

  The lower bound for $I_4$ can be found in a very similar way as for $I_3$.
  Again, we fix the values $u_2,u_3$ such that the polynomials $u_2X^2+X+1$ and $u_3X^3+u_2X^2+X+1$ are irreducible,  and consider the polynomial
  \begin{equation}
  \label{eq:f4}
  f_4 = u_4X^4+ u_3X^3+u_2X^2+X+1, \qquad u_4 \in \Fq^*.
  \end{equation}
  In this case, the discriminant is equal to
  \begin{multline}
    \label{eq:disc_u4}
    \Disc{f_4}= 256 u_{4}^{3} - 192 u_{4}^{2} u_{3} - 128 u_{4}^{2} u_{2}^{2} + 144u_{4}^{2} u_{2} - 27 u_{4}^{2}\\
     + 144 u_{4} u_{3}^{2} u_{2} - 6 u_{4}
    u_{3}^{2} - 80 u_{4} u_{3} u_{2}^{2} + 18 u_{4} u_{3} u_{2} + 16 u_{4}u_{2}^{4}\\
     - 4 u_{4} u_{2}^{3} - 27 u_{3}^{4} + 18 u_{3}^{3} u_{2} - 4
    u_{3}^{3} - 4 u_{3}^{2} u_{2}^{3} + u_{3}^{2} u_{2}^{2}.
  \end{multline}
  In view of the term $256 u_{4}^{3}$ in \eqref{eq:disc_u4}, $\Disc{f_4}$ is not a square polynomial in $u_4$ up to a multiplicative constant.
  Using Lemma \ref{lem:Stick}, when $\Disc{f_4}\ne 0$ (that is, $f_4$ is square-free), we have that
  $\Disc{f_4}$ is not a square element if and only if either $f_4$ is  irreducible, or
  it has two different non-zero roots in $\F_q$ and it is divisible by an irreducible
  polynomial of degree $2$.

  Let $\chi$ be the multiplicative quadratic character of $\F_q$.
  We first count the number of values of $u_4$ such that $\Disc{f_4}$ is non-zero and is not a square element in $\Fq$.
  This number is at least
\begin{equation}
  \label{eq:total}
  \begin{split}
 & \frac{1}{2}\sum_{u_4\in\F_q} \( 1-\chi(\Disc{f_4}) \) - 3/2 - 1 \\
  & \quad \ge \frac{q}{2} - \frac{1}{2} \left | \sum_{u_4\in\F_q} \chi(\Disc{f_4}) \right | - 5/2\\
  & \quad \ge q/2 - \sqrt{q} - 5/2,
  \end{split}
\end{equation}
where the last inequality follows from Lemma~\ref{lem:Weil}.
Note that the term ``$-3/2$'' in \eqref{eq:total} comes from the three possible values of $u_4$ such that  $\Disc{f_4}=0$,
and the term ``$-1$" follows from the fact that we want $u_4 \in \Fq^*$.

 Now, for our purpose, it remains to estimate the number of values of $u_4$ such that the polynomial $f_4$
 has the form
  \begin{equation}
    \label{eq:factorization}
    f_4 = (X+a)(X+b)(cX^2+dX+e),
  \end{equation}
  for some $ a,b,c,d,e\in\Fq$ with $abce\ne 0$ and $a\ne b$, where the polynomial $cX^2+dX+e$ is irreducible.
  If there is no value of $u_4$ satisfying \eqref{eq:factorization}, then this will yield a better bound for $I_4$, which is
  $$
  I_4 \ge I_3 (q/2 - \sqrt{q} - 5/2).
  $$

  In the following, we suppose that there indeed exist values of $u_4$ such that $f_4$ has the form \eqref{eq:factorization}.
 In fact, it is equivalent to count the number of values of $u_4$ such that the reciprocal polynomial
 of $f_4$,
 $$
 g_4= X^4 + X^3 + u_2X^2 + u_3X + u_4,
 $$
 has two different non-zero roots in $\Fq$ and a quadratic irreducible factor.
 Replacing $X$ by $(Y-1/4)$ in $g_4$, we get
 $$
 h_4 =  Y^4 + \al Y^2 + \be Y + \eta,
 $$
 where
   \begin{equation}
  \label{eq:h4}
\left\{\begin{array}{ll}
  \al =u_2 - 3/8,\\
  \be =u_3-u_2/2 + 1/8,\\
  \eta =u_4-u_3/4+u_2/16- 3/256.
\end{array}\right.
\end{equation}
 Then, the cubic resolvent of $h_4$ is
 $$
 R_4= Y^3 +2\al Y^2 +(\al^2-4\eta)Y-\be^2.
 $$
 Since $h_4$ has two roots in $\Fq$, the sum of these two roots is also in $\Fq$.
 By Lemma \ref{lem:resolvent} this means that $R_4$ has a root $y$ which is a square element in $\F_q$,
 where we need to use the assumption that the characteristic of $\Fq$ is not equal to $2$ or $3$.
 If $\be=u_3-u_2/2 + 1/8 \ne 0$, then $y$ is non-zero.
 Note that the number of values of $(u_2,u_3)$ such that $\be = 0$ does not exceed the number of all possible choices of $u_2$ (such that the polynomial $u_2X^2+X+1$ is irreducible), so we have
\begin{equation}
\label{eq:bad u3}
 \# \{(u_2,u_3) \;|\; \be = 0 \} \le I_2/(q-1)^2,
\end{equation}
which implies that
\begin{equation}
\label{eq:good u3}
 \# \{(u_2,u_3) \;|\; \be \ne 0 \} \ge I_3/(q-1)^2 - I_2/(q-1)^2,
\end{equation}

 Now, assume that $\be =u_3-u_2/2 + 1/8 \ne 0$. Since $y \ne 0$ and
 $$
 y^3 +2\al y^2 +(\al^2-4\eta)y-\be^2=0,
 $$
 we obtain
 \begin{equation}
 \label{eq:u4y}
 u_4= \( y^3+2\al y^2+(\al^2+u_3-u_2/4+3/64)y - \be^2 \) / (4y).
 \end{equation}
 So, for each value of $u_4$ satisfying \eqref{eq:factorization}, there exists a square element
$y$ in $\F_q^*$ such that $u_4$ can be recovered by \eqref{eq:u4y}.
Substituting \eqref{eq:u4y} into \eqref{eq:disc_u4}, we get
$$
\Disc{f_4} = t/(4y)^3,
$$
where $t$ is a polynomial in $y$ and has coefficients only depending on $u_2,u_3$.
Note that as a polynomial in $y$, the leading term of $t$ is $256y^9$, and so $\deg t = 9$.

Besides, when $\beta \ne 0$, there is a one-to-one correspondence between values of $u_4$ satisfying \eqref{eq:factorization} and those of $y$.
Because $y \ne 0$, and $\alpha,\beta$ do not depend on $u_4$ when considering the form of $R_4$.

Thus, under the condition $\be  \ne 0$, the number of values of $u_4$ satisfying \eqref{eq:factorization} is at most
\begin{align*}
& \frac{1}{4} \sum_{y\in\F_q^*} \( 1+\chi(y) \)\( 1-\chi(t/(4y)^3) \) \\
& \quad = \frac{1}{4} \sum_{y\in\F_q^*} \( 1+\chi(y) \)\( 1-\chi((4y)^{q-4}t) \) \\
& \quad = \frac{1}{4} \sum_{y\in\F_q^*} \( 1+\chi(y) \)\( 1-\chi(ty) \),
\end{align*}
where the last identity comes from the fact that $q$ is odd and $\chi$ is a multiplicative character.
 Notice that by assumption there already exists a value of $u_4$ such that $\Disc{f_4}$ is not a square element,
 this means that there exists a value of $y$ such that $ty$ is not a square element.
 We also note that the leading term of $ty$ is a square (which is $256y^{10}$).
So, we must have that both $ty$ and $ty^2$ are not a square polynomial in $y$
up to a constant.
Besides, as a polynomial in $y$, each of them has at most 10 distinct roots.
 Now as before, employing Lemma \ref{lem:Weil}, we get
\begin{equation}
\label{eq:bad u4}
\begin{split}
 & \frac{1}{4} \sum_{y\in\F_q^*} \( 1+\chi(y) \)\( 1-\chi(ty) \) \\
 & \quad \le \frac{1}{4} \sum_{y\in\F_q} \( 1+\chi(y) \)\( 1-\chi(ty) \)  \\
 & \quad \le \frac{q}{4} + \frac{1}{4} \left | \sum_{y\in\F_q} \chi(ty) \right |
       + \frac{1}{4} \left | \sum_{y\in\F_q} \chi(ty^2) \right |  \\
 & \quad \le q/4 + 9\sqrt{q}/2,
\end{split}
\end{equation}
where one should note that $t$ is a polynomial in $y$ and $\sum_{y\in \Fq} \chi(y)=0$.

Therefore, combining \eqref{eq:total} with \eqref{eq:bad u4}, fix $u_2,u_3$ such that $\be =u_3-u_2/2 + 1/8 \ne 0$, the number of values of $u_4$ such that $f_4$ is irreducible is at least
$$
q/2 - \sqrt{q} -5/2 - (q/4 + 9\sqrt{q}/2) = q/4 - 11\sqrt{q}/2 -5/2.
$$
So, in view of \eqref{eq:good u3}, we deduce that
$$
I_4 \ge \( I_3/(q-1)^2-I_2/(q-1)^2 \) (q/4 - 11\sqrt{q}/2 -5/2)(q-1)^2,
$$
which, together with Theorem \ref{thm:I2I3}, concludes the proof.
Note that, to ensure $q/4 - 11\sqrt{q}/2 -5/2>0$, we need $q \ge 504$.
\end{proof}

We want to remark that the method we use here will become much more complicated in bounding $I_N$ explicitly for $N\ge 5$, and thus it might be not applicable.

\section{Open Questions }

The results in this paper about consecutive polynomial sequences give some insights to
understand their factorization feature, but definitely there is a long way ahead. Here, we pose some related questions which might be of interest to be
studied. Certainly, there are many other things remaining to be explored.

\begin{question}
Does there exist a consecutive irreducible polynomial sequence $\{ f_n \}$ of infinite length?
\end{question}

In view of the heuristic upper bound of $L(q)$ in \eqref{eq:L(q)} and Table \ref{tab:L(q)}, the answer to this question seems to be no for finite fields.  Unfortunately, this question seems to be beyond reach. Thus, we propose the following
 problem.

\begin{question}
Can one construct a consecutive polynomial sequence $\{ f_n \}$ such that there are infinitely many irreducible polynomials in the sequence?
\end{question}

Here, aside from the existence, we also ask for closed formulas to construct such sequences. The results in~\cite{Ha,Pollack} mentioned before almost show the existence of such sequences
and that the irreducible terms are quite scattered, because in our case we need that all the coefficients are non-zero.  At the end of Section \ref{sec:degFactors}, when $q$ is an odd power of $p$ and under Artin's conjecture, the sequence with all the coefficients equal to 1 contains infinitely many irreducible polynomials. Here, what we want is an unconditional result.

\begin{question}
Is the lower bound for the sequence $\{f_n\}$ of infinite length in Corollary \ref{cor:degree} optimal?
\end{question}

The specific example showed in Theorem \ref{thm:average dep} suggests that maybe the lower bound in Corollary \ref{cor:degree} can be improved.

\begin{question}
Given a consecutive polynomial sequence $\{ f_n \}$, can one find an upper bound of $\omega(f_{m+1}\cdots f_{m+H})$ for $H$ consecutive terms?
\end{question}

Note that for any polynomial $g(X)\in \Fq[X]$ of degree $n\ge 2$, to get large $\omega(g)$ it is required that $g$ only has irreducible factors of low degree. Then, it is easy to check that $\omega(g)\le c(q) n/\log n$, where $c(q)$ is some function with respect to $q$. Thus, for integer $n\ge 2$ we have $\omega(f_n)\le c(q) n/\log n$. Now, the problem is whether we can get better upper bounds for $\omega(f_{m+1}\cdots f_{m+H})$.

We say that a term $f_n$ has a \textit{primitive irreducible divisor} if there exists an irreducible polynomial $g\in \Fq[X]$ such that $g\mid f_n$, but
$g \nmid f_i$ for $i<n$.

\begin{question}
Can one show that almost all terms in $\{f_n\}$ have primitive irreducible divisors?
\end{question}

This question is a natural analogue of the study on the existence of primitive prime divisors in sequences of integers (such as linear recurrences of integers \cite{Bilu, Everest}, and sequences generated in arithmetic dynamics \cite{Ingram, Rice}).

\section*{Acknowledgement}

The authors want to thank the referees for careful reading and valuable comments.
The authors would like to thank Igor Shparlinski for his valuable suggestions and stimulating discussions.
The authors are grateful to the Mathematisches Forschungsinstitut Oberwolfach for hosting them
in a Research in Pairs program.
The research of D. G-P. was supported by the Ministerio de Economia y Competitividad research project MTM2014-55421-P.
The research of A.~O. was supported by the
UNSW Vice Chancellor's Fellowship, and that of M.~S. by the Australian Research Council Grant DP130100237.

\end{document}